\documentclass[preprint]{elsarticle}

\usepackage{amssymb,amsmath,amsthm, color}
\usepackage[margin=2.8cm]{geometry}
\usepackage{url}
\usepackage[utf8]{inputenc}
\newtheorem{theorem}{Theorem}[section]
\newtheorem{lemma}[theorem]{Lemma}
\newtheorem{definition}[theorem]{Definition}
\newtheorem{corollary}[theorem]{Corollary}

\newtheorem{remark}[theorem]{Remark}

\newcommand{\lcm}{\mathrm{lcm}}
\newcommand{\F}{\mathbb F}

\newcommand{\Z}{\mathbb Z}

\newcommand{\doublespace}

\begin{document}

\begin{frontmatter}

\title{Arithmetic constraints of polynomial maps through discrete logarithms}

\author{Lucas Reis}

\ead{lucasreismat@gmail.com}
\address{Departamento de Matem\'{a}tica,
Universidade Federal de Minas Gerais,
UFMG,
Belo Horizonte MG (Brazil),
 30123-970}
\begin{abstract}
Let $q$ be a prime power, let $\F_q$ be the finite field with $q$ elements and let $\theta$ be a generator of the cyclic group $\F_q^*$. For each $a\in \F_q^*$, let $\log_{\theta} a$ be the unique integer $i\in \{1, \ldots, q-1\}$ such that $a=\theta^i$. Given polynomials  $P_1, \ldots, P_k\in \F_q[x]$ and divisors $1<d_1, \ldots, d_k$ of $q-1$, we discuss the distribution of the functions
$$F_{i}:y\mapsto \log_{\theta}P_i(y)\pmod {d_i}, $$
over the set $\F_q\setminus \cup_{i=1}^k\{y\in \F_q\,|\, P_i(y)=0\}$. Our main result entails that, under a natural multiplicative condition on the pairs $(d_i, P_i)$, the functions $F_i$ are asymptotically independent. We also provide some applications that, in particular, relates to past work.
\end{abstract}

\journal{Elsevier}
\end{frontmatter}




\section{Introduction}
Let $q$ be a prime power and $\F_q$ be the finite field with $q$ elements. A polynomial $P\in \F_q[x]$ naturally induces a function from $\F_q$ to itself, namely the evaluation map $y\mapsto P(y)$. Conversely any map from $\F_q$ to itself is uniquely induced by a polynomial in $\F_q[x]$, up to reduction $\pmod{x^q-x}$. Given two ``interesting'' sets $A, B\subseteq \F_q$, it is natural to ask how frequent the image by $P$ of elements in $A$ falls into  $B$. In other words, what is the cardinality of the set $A\cap P^{-1}(B)=\{y\in A\,|\, P(y)\in B\}$. If $A=\F_q$ and $B$ is the set of elements in $\F_q$ satisfying some property $\mathcal P$, the number $\#(A\cap P^{-1}(B))$ measures how frequent $P$ maps elements of $\F_q$ to elements satisfying $\mathcal P$. For instance, if $A=\F_q$ and $B=(\F_q^*)^2=\{y^2\,|\, y\in \F_q^*\}$, $\#(A\cap P^{-1}(B))$ measures how frequent $P$ maps elements of $\F_q$ to nonzero squares. Many authors have studied this kind of problem in the past few years. For instance, in~\cite{bour}, $P(x)=x$, $A$ is a cyclotomic coset in $\F_p^*$, $p$ is prime and $B$ is a symmetric interval. In~\cite{dig1, dig2, dig3, dig4, dig5}, many questions are considered with $\mathcal P$ being related to the digits of elements in $\F_q$ with respect to a given basis (in the sense of Dartyge and S\'arközy~\cite{dig2}).

In this paper we consider an arithmetic setting that naturally arises from the multiplicative structure of $\F_q$. It is known that the multiplicative group $\F_q^*$ is cyclic. If $\theta$ is a generator of $\F_q^*$, for each $a\in \F_q^*$, let $\log_{\theta} a$ be the unique element $i\in \{1, \ldots, q-1\}$ such that $a=\theta^i$. The integer $i=\log_{\theta}a$ is the discrete logarithm of $a$ by $\theta$. We explore the behavior of the function $\log_{\theta}$ modulo divisors of $q-1$, through nonzero polynomial values. For a given divisor $d>1$ of $q-1$ and a polynomial $P\in \F_q[x]$, consider the set $\F_q\cap P^{-1}(\mathcal C_{i, d})$ with $1\le i\le d-1$, where $\mathcal C_{i, d}:=\{\theta^{dj+i}\,|\, 0\le j< (q-1)/d\}$ is a $d$-th cyclotomic coset. In other words,
$$\F_q\cap P^{-1}(\mathcal C_{i, d})=\{y\in \F_q\,|\, \log_{\theta}f(y)\equiv i\pmod d\}.$$
By viewing $P$ as a random map from $\F_q$ to itself,  we would expect the set $\F_q\cap P^{-1}(\mathcal C_{i, d})$ to be of size around $q/d$. More generally, if we consider divisors $1< d_1, \ldots, d_k$ and nonzero polynomials $P_1, \ldots, P_k\in \F_q[x]$, we would expect that the sets 
$$\F_q\cap\left(\bigcap_{\ell=1}^kP_{\ell}^{-1}(\mathcal C_{i_{\ell}, d_{\ell}})\right),$$
have size around $\frac{q}{d_1\ldots d_k}$. This may fail if the polynomials $P_i$ exhibit a ``multiplicative dependence'' with respect to the numbers $d_i$. For instance, if $k=2$, $d_1=d_2=2$ and $P_1\cdot P_2$ is of the form $h(x)^2$ with $h\in \F_q[x]$, we have that $$\F_q\cap P_{1}^{-1}(\mathcal C_{1, 2})\cap P_{2}^{-1}(\mathcal C_{0, 2})=\emptyset.$$
It is natural to further require that the product $d_1\cdots d_k$ is small when compared to $q$, otherwise we do not have ``enough space'' to observe a random behavior. The main result of this paper, Theorem~\ref{thm:main} compiles these ideas and asymptotically confirms these heuristics. The following corollary is a nice application of our main result.

\begin{corollary}\label{cor:main}
Let $f\in \mathbb Z[x]$ be a polynomial not of the form $ag(x)^2$ or $ax\cdot g(x)^2$ with $g\in \mathbb Z[x]$ and $a\in \mathbb Z$, and let $p$ be a prime number. Let $C_{1}$ (resp. $C_{-1}$) be the set of nonzero squares (resp. non squares) $\pmod p$. If $f_p: y\pmod p\mapsto f(y)\pmod p$ is the evaluation map induced by $f$ on $\F_p$ and $n_{i, j}(p)$ denotes the cardinality of the set $f_p^{-1}(C_i)\cap C_j$, then $\lim\limits_{p\to+\infty}\frac{n_{i, j}(p)}{p}=\frac{1}{4}$ for $i,j\in\{1, -1\}$.
\end{corollary}
In the flavor of the previous corollary, Theorem~\ref{thm:ter} entails that for a sufficiently generic polynomial $f\in \mathbb Z[x]$, the polynomial $f(x)\pmod p$ maps around $\varphi(p-1)$ incongruent values modulo $p$ to primitive roots modulo $p$ when $p$ is a large prime number. Another interesting application of our main result, Theorem~\ref{thm:sec} concerns the distribution modulo $d$ of $\log_{\theta} v$ as $v$ runs over nonzero elements of an arbitrary affine subspace of $\F_q$. An immediate application of the former recovers results from~\cite{dig2, dig4, dig5}.  

We comment that the proof of our main result is based on the construction of characteristic functions for cyclotomic cosets by means of multiplicative characters of finite fields. This kind of approach has been extensively used to prove several results on existence and distribution over finite fields; for more details, see~\cite{char} and the references therein.

Here is a summary of the paper. In Section 2 we state our main result, providing remarks and some of its immediate consequences. Section 3 provides background data and in Section 4 we prove our main result. In Section 5 we provide applications of our main result.

\section{Main result}
In this section we state our main result and provide a straightforward application. First, we introduce some notation and useful definitions. Throughout this paper, $q$ denotes a prime power, $\F_q$ is the finite field with $q$ elements and $\theta$ is a generator of $\F_q^*$. For integers $a<b$, we set $[a, b]=\{a, a+1, \ldots, b\}$.
\begin{definition}
Fix $\bold{d}=(d_1, \ldots, d_k)$, where each $d_i>1$ is a divisor of $q-1$. 
\begin{enumerate}[(i)]
\item $\Lambda(\bold{d})$ stands for the set $\prod_{i=1}^k[0, d_i-1]$; 
\item for each $\bold{a}\in \Lambda(\bold{d})$ with $\bold{a}=(a_1, \ldots, a_k)$, we set $\bold{a}_i=a_i$.
\end{enumerate}
The $k$-tuple $\bold{P}=(P_1, \ldots, P_k)$ of nonzero polynomials in $\F_q[x]$ is $\bold{d}$-multiplicatively independent over $\F_q$ if, for $L=\lcm(d_1, \ldots, d_k)$ and $\bold{a}\in \Lambda(\bold{d})$, the equality 
$$P_1^{\frac{L}{d_1}\bold{a}_1}\cdots P_k^{\frac{L}{d_k}\bold{a}_k}=uG(x)^L,$$
with $u\in \F_q^*$ and monic $G\in \F_q[x]$ implies $\bold{a}=(0, \ldots, 0)$ and $G(x)=u=1$.
\end{definition}

\begin{remark}\label{rem:imp}
It follows by the definition that if a $k$-tuple of polynomials in $\F_q[x]$ is $\bold{d}$-multiplicatively independent over $\F_q$, then it is  $\bold{d}$-multiplicatively independent over $\F_{q^t}$ for every $t\ge 1$.
\end{remark}

The main result of this paper is the following theorem.
\begin{theorem}\label{thm:main}
Let $q$ be a prime power and let $\theta$ be a generator of $\F_q^*$. Let $1<d_1,\ldots, d_k$ be divisors of $\F_q$,  $\bold{d}=(d_1, \ldots, d_k)$ and $\bold{a}\in \Lambda(\bold{d})$. Suppose that the $k$-tuple $\bold{P}=(P_1, \ldots, P_k)$ of nonzero polynomials in $\F_q[x]$ is $\bold{d}$-multiplicatively independent over $\F_q$ and let $Z(\bold{P})$ be the number of distinct roots of $P_1\cdots P_k$. Then the number $N(\bold{P},  \bold{a}(\bold{d}), q)$ of elements $y\in \F_q\setminus \cup_{i=1}^k\{y\in \F_q\,|\, P_i(y)=0\}$ such that 
$\log_{\theta}P_i(y)\equiv \bold{a}_i\pmod {d_i}$ for $1\le i\le k$, satisfies
\begin{equation}\label{eq:thm}N(\bold{P},   \bold{a}(\bold{d}), q)=\frac{q}{d_1\cdots d_k}+H(\bold{P}, \bold{d}),\end{equation}
where $|H(\bold{P}, \bold{d})|<Z(\bold{P})\cdot \sqrt{q}$. In particular, $N(\bold{P},   \bold{a}(\bold{d}), q)>0$ if $Z(\bold{P})\le \frac{\sqrt{q}}{ d_1\cdots d_k}$.
\end{theorem}

We observe that the multiplicative independence condition in Theorem~\ref{thm:main} is necessary. Following the notation of Theorem~\ref{thm:main}, if the $k$-tuple $\bold{P}=(P_1, \ldots, P_k)$ of non constant polynomials in $\F_q[x]$ is not $\bold{d}$-multiplicatively independent over $\F_q$, then there exists nonzero $G\in \F_q[x]$, $u\in \F_q^*$ and a non zero vector $\bold{a}\in \Lambda({\bold{d}})$ such that $P_1^{\frac{L}{d_1}\bold{a}_1}\cdots P_k^{\frac{L}{d_k}\bold{a}_k}= uG(x)^L$. If $v=\log_{\theta}u$, we have the following identity for elements $y\in \F_q\setminus \cup_{i=1}^k\{y\in \F_q\,|\, P_i(y)=0\}$: $$\sum_{i=1}^k\frac{L\cdot \bold{a}_i\cdot \log_{\theta}P_i(y)}{d_i}\equiv v\pmod L.$$
For example, if $\bold{a}_1\ne 0$, the former equality entails that $\log_{\theta}P_1(y)\pmod {d_1}$ depends linearly on $v$ and on the values $\log_{\theta}P_i (y)\pmod {d_i}, i\ne 1$, hence it cannot be chosen arbitrarily. In fact one can verify that, in this case, there exist elements  $\bold{a}\in \Lambda(\bold{d})$ with $N(\bold{P},   \bold{a}(\bold{d}), q)=0$, regardless how large is the set  $\F_q\setminus \cup_{i=1}^k\{y\in \F_q\,|\, P_i(y)=0\}$.

Following the notation of Theorem~\ref{thm:main}, we have that $N(\bold{P},   \bold{a}(\bold{d}), q)$ is close to $\frac{q}{d_1\cdots d_k}$, provided that this term dominates the error $H(\bold{P}, \bold{d})$. The following corollary is a straightforward application of Remark~\ref{rem:imp} and Theorem~\ref{thm:main}, and exemplifies this observation. 

\begin{corollary}
Let $q$, $\bold{d}, \bold{a}$ and $\bold{P}$ be as in Theorem~\ref{thm:main}. For each positive integer $t$, let $\theta_t$ be a generator of $\F_{q^t}^*$ and let $N_t$ be the number of elements $y\in \F_{q^t}\setminus \cup_{i=1}^k\{y\in \F_{q^t}\,|\, P_i(y)=0\}$ such that 
$\log_{\theta_t}P_i(y)\equiv \bold{a}_i\pmod {d_i}$ for $1\le i\le k$. Then $N_t=\frac{q^t}{d_1\cdots d_k}(1+o(1))$ as $t\to+\infty$.

\end{corollary}

\section{Preparation}
This section provides the background material required to prove Theorem~\ref{thm:main}. Fix $\theta$ a generator of $\F_q^*$. A multiplicative character of $\F_q$ is a homomorphism $\eta: \F_q^*\to \mathbb C^{\times}$. We observe that if $k\in [0, q-2]$, then the mapping $\eta_k:\theta^a\mapsto e^{\frac{2(ka)\pi i}{q-1}}$ is a multiplicative character of $\F_q$. In fact this describes the entire set $\widehat{\F_q}$ of multiplicative characters of $\F_q$, which is a cyclic multiplicative group of order $q-1$, hence isomorphic to $\F_q^*$. We extend the multiplicative characters to $0\in \F_q$ by setting $\eta(0)=0$ for every $\eta\in \widehat {\F_q}$. For $k\in [0, q-2]$, the multiplicative character $\eta_k$ has order $\frac{q-1}{\gcd(k, q-1)}$. The following lemma provides a formula for the characteristic function of $d$-th powers in $\F_q$ by means of multiplicative characters.

\begin{lemma}\label{lem:char}
If $d$ is a divisor of $q-1$, and $\eta_{\ell(d)}:=\eta_{\frac{(q-1)\ell}{d}}$, then for $a\in \F_q$ the following holds
$$\sum_{\ell=0}^{d-1}\eta_{\ell(d)}(a)=\begin{cases}d&\text{if}\; a=b^d\;\text{for some}\; d\in \F_q^*,\\ 0 & \text{otherwise.}\end{cases}$$
In particular, if $\mathbb I_{a(d)}$ is the characteristic function for the set of elements $\alpha\in \F_q^*$ with $\log_{\theta}\alpha\equiv a\pmod d$, then 
\begin{equation}\label{eq:ent}\mathbb I_{a(d)}(y)=\frac{1}{d}\sum_{j=0}^{d-1}\eta_{j(d)}(y\theta^{-a}).\end{equation}

\end{lemma}

\begin{proof}
If $a=0$, then $\eta_{\ell(d)}(a)=0$ for every $0\le \ell\le d-1$. If $a\ne 0$, write $a=\theta^{dj+r}$, where $0\le r\le d-1$. Therefore, 
$$\sum_{\ell=0}^{d-1}\eta_{\ell(d)}(a)=\sum_{j=0}^{d-1}e^{\frac{2(r\ell)\pi i}{d}}.$$
The former sum equals $d$ if $r=0$ and equals $0$, otherwise. The formula  for $\mathbb I_{a(d)}$ follows directly from the fact that $\log_{\theta}\alpha\equiv a\pmod d$ if and only if $\alpha\theta^{-a}$ is a nonzero $d$-th power in $\F_q$.
\end{proof}
The following character sum estimate is useful.

\begin{theorem}[see Theorem 5.41 of~\cite{LN}]\label{Gauss1} Let $\eta$ be a multiplicative character of $\F_{q}$ of order $r>1$ and $F\in \F_{q}[x]$ be a polynomial of positive degree such that $F$ is not of the form $ag(x)^r$ for some $g\in \F_{q}[x]$ with degree at least $1$ and $a\in \F_q$. Suppose that $z$ is the number of distinct roots of $F$ in its splitting field over $\F_{q}$. Then the following holds: 
$$\left|\sum_{c\in \F_{q}}\eta(F(c))\right|\le (z-1)\sqrt{q}.$$
\end{theorem}

Further applications of Theorem~\ref{thm:main} require the following auxiliary results.

\begin{lemma}\label{lem:essential}
Let $d>1$ be a positive integer and let $f\in \mathbb Z[x]$ be a polynomial not of the form $ag(x)^d$ with $a\in \mathbb Z$ and $g(x)\in \mathbb Z[x]$. Then for every sufficiently large prime $p$, the polynomial $f(x)\pmod p$ is not of the form $ag(x)^d$ with $a\in \F_p$ and $g\in \F_p[x]$.\end{lemma}

\begin{proof}
The result is straightforward if $f$ has degree at most $d-1$. Suppose that $f$ has degree $m\ge d>1$ and let $t$ be the largest multiplicity among its roots over $\mathbb C$. By writing $f(x)=A(x)B(x)^d$ with $A, B\in \mathbb Z[x]$, it suffices to consider the case where $t\le d-1$. Let $f_1, \ldots, f_r$ be the distinct irreducible divisors of $f$ over $\Z[x]$ and set $S=\max\limits_{1\le i<j\le r}|R(f_i, f_j)|$, where $R(f_i, f_j)\in \mathbb Z$ is the resultant of $f_i$ and $f_j$. From construction, the polynomials $f_i$ do not have common roots and so $R(f_i, f_j)\ne 0$ for every $1\le i<j\le r$. In particular, for any prime $p>S$, the polynomials $f_i\pmod p$ are pairwise relatively prime. Therefore, $f(x)\pmod p$ is of the form $ag(x)^d$ with $a\in \F_p$ and $g\in \F_p[x]$ if and only if each polynomial $f_i(x)\pmod p$ has the same form. From this observation, it suffices to consider the case where $f\in \mathbb Z[x]$ is of the form $a\cdot F(x)$ with $F\in \Z[x]$ being a power of a monic irreducible polynomial in $\mathbb Z[x]$. In this case, since $t$ is the largest multiplicity of a root of $f$, the $t$-th derivative $f^{(t)}$ of $f$ does not have a common root with $f$. In particular, $T=R(f, f^{(t)})$ is a nonzero integer. However, if $f_p(x):=f(x)\pmod p$ is of the form $ag(x)^d$ with $a\in \F_p$ and $g(x)\in \F_p[x]$, it follows that $R(f_{p}, f_p^{(t)})=0\in \F_p$. The latter is equivalent to $T\equiv 0\pmod p$. Since $T\ne 0$, we have that $p\le |T|$ and the result follows. 
\end{proof}

\begin{lemma}\label{lem:vec}
Let $q$ be a prime power, let $n$ be a positive integer and let $V\subseteq \F_{q^n}$ be an $\F_q$-affine space of dimension $t\ge 1$. Then there exists a separable polynomial $L=L_{V}$ of degree $q^{n-t}$ such that $L(\F_{q^n})=V$ and, for each $v\in V$, the equation $L(x)=v$ has exactly $q^{n-t}$ solutions in $\F_{q^n}$.
\end{lemma}

\begin{proof}
With no loss of generality we can suppose that $V$ is an $\F_q$-vector space. Set $M(x)=\sum_{v\in V}(x-v)$. It follows by induction on $t$ that $M(x)$ is of the form $\sum_{i=0}^tb_ix^{q^i}$. In particular, $M(a+b)=M(a)+M(b)$ and  $M(\alpha a)=\alpha M(a)$ for every $a, b\in \F_{q^n}$ and every $\alpha\in \F_q$. In other words, $c\mapsto M(c)$ is an $\F_q$-linear map on $\F_{q^n}$. The Rank-Nullity Theorem entails that $U=M(\F_{q^n})\subseteq \F_{q^n}$ is an $\F_q$-vector space of dimension $n-t$. In particular, $L(x)=\prod_{u\in U}(x-u)$ is a polynomial of the form $\sum_{i=0}^{n-t}c_ix^{q^i}$ and $L(M(y))=0$ for every $y\in \F_{q^n}$. Since $L(M(x))$ is monic of degree $q^n$, we conclude that 
$$L(M(x))=x^{q^n}-x.$$ In particular, $L(M(L(x)))=L(x)^{q^n}-L(x)$. Since $L(x)$ is of the form $\sum_{i=0}^{n-t}c_ix^{q^i}$ with $c_i\in \F_{q^n}$, it follows that $L(x)^{q^n}-L(x)=L(x^{q^n}-x)$. In other words, $L(T(x))=L(S(x))$ with $T(x)=L(M(x))$ and $S(x)=M(L(x))$. Hence $L(M(x))=M(L(x))$, i.e., $M(L(x))=x^{q^n}-x$. Therefore, $L$ has degree $q^{n-t}$ and $L(\F_{q^n})=V$. From construction, $L$ is separable.  Moroever, since $L(x)$ is of the form $\sum_{i=0}^{n-t}c_ix^{q^i}$, we have that $c\mapsto L(c)$ is an $\F_q$-linear map on $\F_{q^n}$ whose kernel over $\F_{q^n}$ is an $\F_q$-vector space of dimension $n-t$. In particular, for each $v\in V$, the  equation $L(x)=v$ has exactly $q^{n-t}$ solutions in $\F_{q^n}$.

\end{proof}

\section{Proof of Theorem~\ref{thm:main}}
Following the notation of Theorem~\ref{thm:main} and Lemma~\ref{lem:char}, Eq.~\eqref{eq:ent} entails that the number $N(\bold{P},   \bold{a}(\bold{d}), q)$ of elements $y\in \F_q\setminus \cup_{i=1}^k\{y\in \F_q\,|\, P_i(y)=0\}$ such that 
$\log_{\theta}P_i(y)\equiv \bold{a}_i\pmod {d_i}$ for $1\le i\le k$, satisfies the following equality
$$N(\bold{P},   \bold{a}(\bold{d}), q)=\sum_{y\in \F_q}\prod_{i=1}^k\mathbb I_{\bold{a}_i(d_i)}(P_i(y))=\sum_{y\in \F_q}\prod_{i=1}^k\frac{1}{d_i}\left(\sum_{j=0}^{d_i-1}\eta_{j(d_i)}(P_i(y)\theta^{-\bold{a}_i})\right).$$
Therefore, 
\begin{equation}\label{eq:sum}(d_1\cdots d_k)\cdot N(\bold{P},   \bold{a}(\bold{d}), q)=\sum_{\bold{c}\in \Lambda({\bold d})}\omega_{\bold {c}, \theta}\cdot S_{\bold{c}, q},\end{equation}
where $S_{\bold{c}, q}=\sum_{y\in \F_q}\left(\prod_{j=1}^k\eta_{\bold{c}_j(d_j)}(P_j(y))\right)$ and $\omega_{\bold{c}, \theta}=\prod_{j=1}^k\eta_{\bold{c}_j(d_j)}(\theta^{-\bold{a}_i})$. For $L=\lcm(d_1, \ldots, d_k)$ and $y\in \F_q$, we have that 
$$\prod_{j=1}^k\eta_{\bold{c}_j(d_j)}(P_j(y))=\prod_{j=1}^k\eta_{\frac{(q-1)\bold{c}_j}{d_j}}(P_j(y))=\eta_{\frac{q-1}{L}}\left(\prod_{j=1}^kP_j(y)^{\frac{L\bold{c}_j}{d_j}}\right),$$
with the convention that $0^0=0$. We observe that $\eta_{\frac{q-1}{L}}$ is a character of order $L$. From hypothesis, the $k$-tuple $(P_1, \ldots, P_k)$ is $\bold{d}$-multiplicatively independent and so the polynomial $$P_{\bold{c}}:=\prod_{j=1}^kP_j(x)^{\frac{L\bold{c}_j}{d_j}},$$ is of the form $u\cdot G(x)^L$ if and only if $\bold{c}=\bold{0}:=(0, \ldots, 0)$. For $\bold{c}=\bold{0}$, we have that $\omega_{\bold{c}, \theta}=1$ and the equality $P_{\bold{c}}(y)=0$ has at most $Z(\bold{P})$ solutions $y\in \F_q$. Therefore, $$\omega_{\bold{0}, \theta}\cdot S_{\bold{0}, q}\ge q-Z(\bold{P}).$$ 
From construction, the numbers $\omega_{\bold{c}, \theta}$ are of norm $1$. In particular, for every element $\bold{c}\in \Lambda{(\bold{d})}\setminus \{\bold{0}\}$, Theorem~\ref{Gauss1} entails that 
$|\omega_{\bold{c}, \theta}\cdot S_{\bold{c}, q}|\le (Z(\bold{P})-1)\sqrt{q}$. Taking estimates in Eq.~\eqref{eq:sum} we obtain that $N(\bold{P},   \bold{a}(\bold{d}), q)=\frac{q}{d_1\cdots d_k}+H(\bold{P}, \bold{d})$, where 
\begin{align*}|H(\bold{P}, \bold{d})|\le \frac{Z(\bold{P})+(d_1\cdots d_k-1)(Z(\bold{P})-1)\sqrt{q}}{d_1\cdots d_k}<Z(\bold{P})\cdot \sqrt{q}.\end{align*}

\section{Applications of Theorem~\ref{thm:main}}
Here we provide some applications of Theorem~\ref{thm:main}, including the proof of Corollary~\ref{cor:main}.  The following corollary entails that, under some conditions on $P\in \F_q[x]$, we can guarantee that $P$ takes nonzero square values when evaluated at $p$ consecutive elements in $\F_q$, where $p$ is the characteristic of $\F_q$.

\begin{corollary}\label{cor:1}
Let $q$ be a power of an odd prime $p$ and $P\in \F_q[x]$ a non constant polynomial with $t>0$ distinct roots such that $\gcd(P(x), P(x+i))=1$ for every $0< i\le p-1$. Furthermore, assume that $P(x)$ is not of the form $ag(x)^2, a\in \F_q$. If $t\cdot p\cdot 2^p\le \sqrt{q}$, there exists $u\in \F_q$ such that the elements $$P(u), P( u+1), \ldots, P(u+p-1),$$ are all nonzero squares in $\F_q$. In particular, if $\sqrt{q}>p\cdot 2^p$, there exists $u\in \F_q$ such that $u, u+1, \ldots, u+p-1$ are all nonzero squares in $\F_q$.
\end{corollary}

\begin{proof}
Applying Theorem~\ref{thm:main} for $k=p$, $\bold{d}=(2, \ldots, 2)$, $\bold{a}=(0, \ldots, 0)$ and $\bold{P}=(P(x), P(x+1), \ldots, P(x+p-1))$, it suffices to prove that $\bold{P}$ is $\bold{d}$-multiplicatively independent. If this was not the case, there would exist integers $0\le i_1<\ldots< i_s\le p-1$, $G\in \F_q[x]$ and $u\in \F_q$ such that
$P(x+i_1)\cdots P(x+i_s)=uG(x)^2$. From hypothesis, the polynomials $P(x+i_j)$ are pairwise relatively prime and so the latter implies that each $P(x+i_j)$ is of the form $ag(x)^2$, a contradiction. The second statement follows from the first by taking $P(x)=x$.
\end{proof}

\subsection{Proof of Corollary~\ref{cor:main}} 
We observe that the degree of $f$ is fixed and $p$ goes to infinity. Applying Theorem~\ref{thm:main} for $k=2$, $\bold{d}=(2, 2)$ and $\bold{a}=(r, s)$ with $r, s\in \{0, 1\}$ and $\bold{f}_p=(x, f(x)\pmod p)$, it suffices to prove that, for sufficiently large $p$, the following hold: 
\begin{itemize}
\item $f(x)\pmod p$ does not vanish;
\item $\bold{f}_p$ is $\bold{d}$-multiplicatively independent. 
\end{itemize}

If $\ell\in \mathbb Z$ is the leading coefficient of $f$, we observe that $f(x)\pmod p$ does not vanish for every $p>|\ell|$. Moreover, by the definition, the pair $$\bold{f}_p=(x, f(x)\pmod p),$$ is not $\bold{d}$-multiplicatively independent if and only if either $f(x)\pmod p$ or $x\cdot f(x)\pmod p$ is of the form $uG(x)^2$ with $u\in \F_p^*$ and $G\in \F_p[x]$. From Lemma~\ref{lem:essential}, the latter cannot occur for large primes $p$ unless $f(x)$ or $x\cdot f(x)$ is of the form $ah(x)^2$ with $h(x)\in \mathbb Z[x]$ and $a\in \mathbb Z$. Equivalently, $f(x)$ is of the form $bg(x)^2$ or $bxg(x)^2$ with $g(x)\in \mathbb Z[x]$ and $b\in \mathbb Z$, a contradiction with our hypothesis.

\subsection{Primitive roots as polynomial values}
Let $p$ be a prime number. An element $g\in \mathbb F_p$ is a primitive root modulo $p$ if it generates the cyclic group $\mathbb F_p^*$ or, equivalently, $g^t\equiv 1\pmod p$ if and only if $t\equiv 0\pmod{p-1}$. It is known that the number of such elements equals $\varphi(p-1)$, where $\varphi$ is the Euler totient function. As an application of Theorem~\ref{thm:main} we show that, for large primes $p$ and a sufficiently generic polynomial $f(x)\in \mathbb Z[x]$,  the function $f_p:z\pmod p\to f(z)\pmod p$ maps around $\varphi(p-1)$ incongruent values $\pmod p$ to primitive roots $\pmod p$.

\begin{theorem}\label{thm:ter}
Let $f(x)\in \mathbb Z[x]$ be a polynomial not of the form $ag(x)^d$ with $a\in \mathbb Z, g\in \mathbb Z[x]$ and $d>1$. For a prime $p$, let $g_{p, f}$ be the number of elements $y\in \mathbb F_p$ such that $f(y)\pmod p$ is a primitive root modulo $p$. Then, as $p\to +\infty$, we have that 
$$g_{p, f}=\frac{p\cdot \varphi(p-1)}{p-1}(1+o(1)).$$
\end{theorem}

\begin{proof}
Let $d\ge 1$ be the degree of $f$. From Lemma~\ref{lem:essential}, there exists $M>0$ such that $f_p(x):=f(x)\pmod p$ is a polynomial of degree $e=\deg(f)$, not of the form $b\cdot h(x)^d$ with $b\in \F_p, g(x)\in \F_p[x]$ and $d>1$ whenever $p>M$. 

From now and on, we only consider primes $p>M$. Let $d_1<\ldots< d_s$ be the distinct prime divisors of $p-1$ and $D:=d_1\cdots d_s$. Let $g$ be any primitive root modulo $p$. We observe that $y\in \F_p^*$ is another primitive root modulo $p$ if and only if $\log_gy\not\equiv 0\pmod{d_i}$ for each $1\le i\le s$. 
In the notation of Theorem~\ref{thm:main}, for each divisor $\omega>1$ of $D$, set $\theta=g$, $k=1$, $\bold{d}_{\omega}=(\omega)$, $\bold{a}_{\omega}=(0)$ and $\bold{P}_{\omega}=(f_p)$. By the previous observations and a simple inclusion-exclusion argument, we have that
\begin{equation}\label{eq:prim}g_{p, f}=\sum_{y\in \F_p\atop f(y)\ne 0}1+\sum_{\omega|D\atop \omega\ne 1}N(\bold{P},   \bold{a}_{\omega}(\bold{d}_{\omega}), p)\mu(\omega),\end{equation}
where $\mu$ is the Moebius function over the integers. From hypothesis, $\bold{P}_{\omega}$ is $\bold{d}_{\omega}$-multiplicatively independent and, since $f_p$ has degree $e$, it follows that $Z(\bold{P}_{\omega})\le e$. Therefore, Eq.~\eqref{eq:prim} and Theorem~\ref{thm:main} entail that
$$g_{p, f}=p-E+\sum_{\omega|D\atop \omega\ne 1}\mu(\omega)\frac{p}{\omega}+\sum_{\omega|D\atop \omega\ne 1}\mu(\omega)\cdot R_{\omega},$$
where $0\le E\le e$ is the number of roots of $f_p$ over $\F_p$ and $R_{\omega}$ is a real number with $|R_{\omega}|\le e\sqrt{p}$. Since $\sum_{t|n}\frac{\mu(t)}{t}=\frac{\varphi(n)}{n}$, we obtain that 
$$\left|g_{p, f}-\frac{p\cdot \varphi(p-1)}{p-1}\right|\le (2^s-1)d\sqrt{p}+E<2^se\sqrt{p}.$$
For every $\varepsilon>0$, we have the well-known bounds $2^s=o(p^{\varepsilon})$ and $\varphi(p-1)\gg p^{1-\varepsilon}$, from where the result follows.
\end{proof}

We comment that the condition on $f$ in Theorem~\ref{thm:ter} is natural. For instance, if $f(x)=ag(x)^d$ for some $d>1$ and $p$ is a prime with $p\equiv 1\pmod d$ such that $a^{\frac{p-1}{d}}\equiv 1\pmod p$, then $f(y)\pmod p$ is a perfect $d$-th power in $\F_p$ for every $y\in \F_p$. Therefore, no value of $f(y)\pmod p$ can be a primitive root $\pmod p$. However, if we allow $f$ to be of the form $ag(x)^{d_0}$ for some $d_0>1$, Theorem~\ref{thm:ter} remains true for primes $p$ with $\gcd(p-1, d_0)=1$. We omit details.

\subsection{Discrete logarithm in affine subspaces}

Write $q=p^n$ and let $\theta$ be a generator of $\F_q^*$. Let $V\subseteq \F_q$ be a $t$-dimensional $\F_p$-affine space with $t\le n$. We observe that $V$ has an additive structure, while cyclotomic cosets have multiplicative structure. Therefore, we may expect equal distribution modulo $d$ of the values $\log_{\theta}v$ for nonzero elements $v\in V$, where $d>1$ is a divisor of $q-1$. Of course this fails if the size of $V$ is not large when compared to $d$ or if $V$ has also a multiplicative structure. In fact, we can even observe classes modulo $d$ not being reached by elements in $V$ through the discrete logarithm. For instance, if $p$ is odd and $n=2r$, then $V=\F_{p^r}=\{0\}\cup \{\theta^{(p^r+1)j}\,|\, 0\le j\le p^r-1\}$ and so $\log_{\theta}v$ is always even, i.e., every nonzero element of $V$ is a square in $\F_q$. As an application of Theorem~\ref{thm:main}, we provide a sufficient condition on $k$ and $d$ in order to $\log_{\theta}v$  intersect, at least once, each class modulo $d$. If $p^t$ is large when compared to $d\sqrt{q}$, we observe an equidistribution phenomena, i.e., each class modulo $d$ is realized by around $p^t/d$ elements in $V$. This is compiled in the following theorem.

\begin{theorem}\label{thm:sec}
 Let $V\subseteq \F_q$ be a $t$-dimensional $\F_p$-affine space and  let $d$ be a divisor of $q-1$, where $q=p^n$. If $a\in\{0, \ldots, d-1\}$ and $V_{a(d)}$ denotes the number of nonzero $v\in V$ such that $\log_{\theta}v\equiv a\pmod d$, then $$V_{a(d)}=\frac{p^t}{d}+r(a, d),$$ where $|r(a, d)|< p^{n/2}$. 
In particular, if $p^{t-n/2}\ge d$, we have that $V_{a(d)}>0$. 
\end{theorem}
\begin{proof}
From Lemma~\ref{lem:vec}, there exists a separable polynomial $f$ of degree $p^{n-t}$ such that $f(\F_q)=V$ and, for each $v\in V$, the equation $f(x)=v$ has exactly $p^{n-t}$ solutions in $\F_q$. In particular, $V_{a(d)}$ equals $\frac{1}{p^{n-t}}$ times the number of elements $y\in \F_q$ such that $\log_{\theta}f(y)\equiv a\pmod d$. In the notation of Theorem~\ref{thm:main}, set $k=1$, $\bold{P}=(f)$, $\bold{d}=(d)$ and $\bold{a}=(a)$. Therefore, $Z(\bold{P})=p^{n-t}$ and $V_{a(d)}=\frac{N(\bold{P},   \bold{a}(\bold{d}), q)}{p^{n-t}}$. Since the polynomial $f$ is separable, $\bold{P}$ is $\bold{d}$-multiplicatively independent and the result follows by Eq.~\eqref{eq:thm} in Theorem~\ref{thm:main}.
\end{proof}
\begin{remark}
We observe that Theorem~\ref{thm:sec} is non trivial only for $\F_p$-affine spaces in $\F_q$ with dimension $t>n/2$. However, our example where $q=p^{2r}$ is odd, $V=\F_{p^{r}}$, $d=2$ and $a=1$ entails that, in general, this range is the best possible. In this context, Theorem~\ref{thm:sec} is sharp with respect to $t$. 
\end{remark}

\subsubsection{Perfect powers and digits}

In~\cite{dig2}, the authors introduce the notion of digits of elements in $\F_q$. If $\mathcal B=\{b_1, \ldots, b_n\}$ is and $\F_p$-basis for $\F_q$, regarded as an $\F_p$-vector space, any element $a\in \F_q$ is written uniquely as $a=\sum_{i=1}^na_ib_i$ with $a_i\in \F_p$. The elements $a_1, \ldots, a_n$ are the digits of $a$ in the basis $\mathcal B$. We write $s_{\mathcal B}(a)=\sum_{i=1}^na_i$, the sum of the digits of $a$ in the basis $\mathcal B$. According to Theorem 1 in~\cite{dig2}, for each $c\in \F_p$, the number $n_{c}$ of squares $y^2, y\in \F_q^*$ such that $s_{\mathcal B}(y^2)=c$ satisfies 
$n_c=\frac{p^{n-1}}{2}+h_c$, where $|h_c|\le p^{n/2}$. We observe that, by the definition, the set $S_c$ of elements $a\in \F_q$ satisfying $s_{\mathcal B}(a)=c$ comprises an $\F_p$-affine space of dimension $t=n-1$. In this context, Theorem 1 in~\cite{dig2} follows by Theorem~\ref{thm:sec} with $V=S_c$, $d=2$ and $a=0$. More generally, we have the following result.

\begin{corollary}\label{cor:puf}
Fix $\mathcal B$ an $\F_p$-basis of $\F_q$ and $d>1$ a divisor of $q-1$. For each $c\in \F_p$, let $n_{c, d}$ be the number of $d$-th powers $y^d, y\in \F_q^{*}$ whose sum of digits in the basis $\mathcal B$ satisfies $s_{\mathcal B}(y^d)=c$. Then 
$$n_{c, d}=\frac{p^{n-1}}{d}+h_{c, d},$$
where $|h_{c, d}|\le p^{n/2}$. 
\end{corollary}
A variation of the previous corollary is obtained in Corollary~1.3 of~\cite{dig5}, where the author provides a slightly better bound for the error $h_{c, d}$. This improvement is obtained by a more detailed estimate on certain character sums. 

In~\cite{dig4} the author explores the distribution of polynomial values with some digits prescribed. In particular, she obtains an estimate $\frac{p^{n-k}}{2}+r_{k}$ for the number of squares $y^2, y\in \F_q$ with exactly $k$ digits prescribed: see Theorem~1.3 of~\cite{dig4} for more details. We easily verify that the set of elements $a\in \F_q$ with $k$ digits prescribed with respect to an $\F_p$-basis comprises an $\F_p$-affine space of dimension $t=n-k$. In particular, applying Theorem~\ref{thm:sec} as in Corollary~\ref{cor:puf}, we can recover an estimate $\frac{p^{n-k}}{2}+r'_{k}$ with the bound $|r'_k|\le p^{n/2}$,  which is slightly weaker than the one in Theorem~1.3 of~\cite{dig4}. Again, this is due to a more detailed estimate on certain character sums employed there. 

\section*{Acknowledgments}
The author was partially supported by PRPq/UFMG (ADRC 09/2019).

\end{document}